\documentclass[a4paper]{amsart}

\usepackage{amsmath,amssymb,amsthm,a4wide}

\usepackage{tikz}
\usetikzlibrary{calc}
\usetikzlibrary{arrows}

\usepackage{hyperref,caption}
\usepackage{graphicx}
\usepackage{graphics}

\newtheorem{theorem}{Theorem}

\newtheorem*{corollary}{Corollary}
\newtheorem*{proposition}{Proposition}

\newcommand{\diam}{\operatorname{diam}}

\begin{document}

\title[]{on the diffusion geometry of graph\\ laplacians and applications}

\author[]{Xiuyuan Cheng}
\address[Xiuyuan Cheng]{Applied Mathematics Program, Yale University, New Haven, CT 06510, USA}
\email{xiuyuan.cheng@yale.edu}

\author[]{Manas Rachh}
\address[Manas Rachh]{Applied Mathematics Program, Yale University, New Haven, CT 06510, USA}
\email{manas.rachh@yale.edu}

\author[]{Stefan Steinerberger}
\address[Stefan Steinerberger]{Department of Mathematics, Yale University, New Haven, CT 06510, USA}
\email{stefan.steinerberger@yale.edu}

\begin{abstract} 
We study directed, weighted graphs $G=(V,E)$  and 
consider the (not necessarily symmetric) 
averaging operator 
$$ (\mathcal{L}u)(i) = -\sum_{j \sim_{} i}{p_{ij} (u(j) - u(i))},$$
where $p_{ij}$ are normalized edge weights. 
Given a vertex $i \in V$, we define the diffusion distance to a set $B \subset V$ as the smallest 
number of steps $d_{B}(i) \in \mathbb{N}$
required for half of all random walks started in $i$ and moving 
randomly with respect to the weights $p_{ij}$ to visit $B$ within
$d_{B}(i)$ steps.
Our main result is that the eigenfunctions interact nicely with 
this notion of distance.
In particular, if $u$ satisfies 
$\mathcal{L}u = \lambda u$ on $V$ and 
$$ B = \left\{ i \in V:    - \varepsilon   \leq u(i) \leq \varepsilon   \right\} \neq \emptyset,$$
then, for all $i \in V$,
$$ d_{B}(i) \log{\left(  \frac{1}{|1-\lambda|} \right) } 
\geq    \log{\left( \frac{ |u(i)| }{\|u\|_{L^{\infty}}}  \right)}  -  
\log{\left(\frac{1}{2} + \varepsilon\right)}.$$
$d_B(i)$ is a remarkably good approximation of $|u|$ in the sense of
having very high correlation.
The result implies that the 
classical one-dimensional spectral embedding preserves 
particular aspects of geometry in the presence of clustered data.
We also give a continuous variant of the result which
has a connection to the hot spots conjecture. 
\end{abstract}
\maketitle

\section{Introduction and main result}
\subsection{Introduction.} 
This paper is motivated by the following continuous problem: 
let $\Omega \subset \mathbb{R}^2$ be a compact set with smooth boundary
and consider an eigenfunction of the Laplacian
$$ -\Delta u = \lambda u$$
with Dirichlet boundary conditions on $\partial \Omega$. 
A natural question (also for more general elliptic equations) is whether 
it is possible to
specify in advance and using only the geometry of $\Omega$ where the maximum 
is located. 
A recent result of the second and third author \cite{manas} shows the sharp 
result that the location of the maximum is at 
least $c\cdot \lambda^{-1/2}$ away
from the boundary when the domain $\Omega$ is simply
connected. 
Here $c$ is a universal constant, see \cite{ban2, stein1} for a detailed
discussion. 
If we consider Neumann conditions on the boundary $\partial u/\partial n = 0$, 
then it follows from the 
hot spots conjecture of Rauch \cite{burd2} that 
both maximum and minimum of the eigenfunction should be assumed on
the boundary.
This conjecture is known \cite{burd} to fail in general 
but is widely believed to be true at least for convex domains. 
The purpose of our paper is to prove an optimal inequality for a related 
problem on graphs.
We derive a certain type of guarantee that spectral 
clustering is well-behaved in the presence of clusters.
Furthermore, we prove a continuous version of our result
wherein we show that the location of both maximum and minimum of the 
first nontrivial Neumann eigenfunction for Laplace's equation is not too 
close to the nodal line $\left\{x: u(x) = 0\right\}$.

\subsection{Setup}
 Let $B = (V,E)$ be a connected, directed, weighted graph with 
 normalized edge weights $p_{ij}$ 
$$ \sum_{j = 1}^{n}{p_{ij}} = 1.$$
 We introduce a Laplacian-type operator $\mathcal{L}$ acting on functions $u: V \rightarrow \mathbb{R}$
as 
$$ (\mathcal{L}u)(i) = -\sum_{j =1}^{n}{p_{ij}(u(j) - u(i))}.$$
Note that, at this level of generality, the operator $\mathcal{L}$ need 
not be self-adjoint.
This formulation is motivated by the mean-value property of the 
Laplacian in the continuous setting, the negative sign ensures that the 
operator is positive definite.
We additionally assume, to avoid nontrivial counterexamples, that a random walk (with transition probabilities $p_{ij}$) can ultimately travel from every vertex to every other vertex. 
There are naturally associated eigenvalues $u$ which satisfy
$$
 \mathcal{L} u = \lambda u ~\quad \mbox{on}~V 
$$
It is easy to see (using either the Perron-Frobenius theorem or the 
Gershgorin circle theorem) that $|1-\lambda| < 1$. 
Our main statement is that if $|u(i)|$ is large, then $i$ is `far away' 
from the set of vertices where $|u|$ is small, where the notion of 
distance is defined below.

\subsection{Diffusion distance to the boundary.} 
Let $x:\mathbb{N} \to V$ denote a random walk associated with a 
Markov chain on the graph with transition probabilities given by
$$ \mathbb{P}\left( x(t+1) = j ~ \big| ~ x(t) = i\right) = p_{ij} \, . $$
For given $i \in V$ the {\em diffusion distance} 
$d_{B}(i)$ is defined as the smallest integer $k \in \mathbb{N}$ such 
that the likelihood of a random walk started in $i$ is visiting $B$ within $k$ time steps is atleast $1/2$
$$ d_{B} (i) = \inf_{} \left\{ k  \in \mathbb{N}: 
\mathbb{P}( \exists ~k'<k \, , x(k') \in B ~ \big| ~ x(0)=i) \geq 
\frac{1}{2} \right\}.$$

Our setup implies that the diffusion distance $d_{B}$ is always finite.
For example, consider the standard random walk on $\mathbb{Z}$ 
and let $B=\{1,10 \}$ (Figure~\ref{fig:p10exmp}).
We see 
that the diffusion distance has `quadratic growth away from the boundary' 
due to the fact that a random walk on $\mathbb{Z}$ only 
travels up to distance $\sim \sqrt{k}$ in $k$ steps.

\begin{center}
\begin{figure}[h!]
\begin{tikzpicture}[scale =1.2]
\draw [thick] (0,0) -- (9,0);
\filldraw (0,0) circle (0.08cm);
\draw [ thick] (-0.2, -0.2) -- (-0.2, 0.2);
\draw [ thick] (-0.2, -0.2) -- (0.2, -0.2);
\draw [ thick] (0.2, -0.2) -- (0.2, 0.2);
\draw [ thick] (-0.2, 0.2) -- (0.2, 0.2);
\node at (0, -0.5) {$B$};
\node at (-0.5, 0.5) {$d_{B}$};
\node at (0, 0.5) {$0$};
\filldraw (1,0) circle (0.08cm);
\node at (1, 0.5) {$1$};
\filldraw (2,0) circle (0.08cm);
\node at (2, 0.5) {$8$};
\filldraw (3,0) circle (0.08cm);
\node at (3, 0.5) {$13$};
\filldraw (4,0) circle (0.08cm);
\node at (4, 0.5) {$15$};
\filldraw (5,0) circle (0.08cm);
\node at (5, 0.5) {$15$};
\filldraw (6,0) circle (0.08cm);
\node at (6, 0.5) {$13$};
\filldraw (7,0) circle (0.08cm);
\node at (7, 0.5) {$8$};
\filldraw (8,0) circle (0.08cm);
\node at (8, 0.5) {$1$};
\filldraw (9,0) circle (0.08cm);
\node at (9, 0.5) {$0$};
\draw [ thick] (9-0.2, -0.2) -- (9-0.2, 0.2);
\draw [ thick] (9-0.2, -0.2) -- (9+0.2, -0.2);
\draw [ thick] (9.2, -0.2) -- (9.2, 0.2);
\draw [ thick] (9-0.2, 0.2) -- (9.2, 0.2);
\node at (9, -0.5) {$B$};
\end{tikzpicture}
\caption{Diffusion distance from $B=\left\{1,10\right\}$ for the standard
random walk on $\mathbb{Z}$.}
\label{fig:p10exmp}
\end{figure}
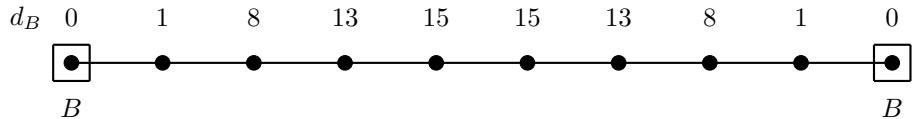
\end{center}
\vspace{-10pt}

The notion is uniquely suited for graphs: if a graph is a 
fine discretization of a convex domain $\Omega \subset \mathbb{R}^n$, then the
diffusion distance scales like a rescaling of the squared distance to the boundary and this scaling is independent of the dimension $n$ (except for constants),
however, a general graph may not have constant `dimensionality' and the diffusion distance is naturally adaptive. In the continuous
case, it is also highly related to the notions of capacity and harmonic measure.

\subsection{Spectral embedding} Let us  
assume that the graph $G$ essentially decomposes into two
clusters of roughly equal size connected via a bottleneck (see Fig. 2). 
Classical intuition suggests that the first (nontrivial) eigenfunction $u$ 
(associated to the largest real eigenvalue)
of $\mathcal{L}$ will be negative on one cluster, positive on the other 
cluster and $\sim 0$ in the bottleneck -- this intuition has been made precise in a variety of different ways (most famously in Cheeger's inequality). In particular, the map
$$ u:V \rightarrow \mathbb{R}$$
can be understood as a classifier: for any element $i \in V$, the sign of 
$u(i)$ allows to determine the cluster which contains $i$. 
Or, put differently, $u$ is effective in isolating
this basic geometric feature. 
However, one would naturally like to go further and argue that, while the sign 
of $u(i)$ determines the cluster, the magnitude $|u(i)|$ 
should be able to serve as a quantitative measure of certainty of that 
estimate. In particular, the value $i$ for which $u$ 
assumes its minimum should be the most typical
representitive of its cluster that is most easily distinguished from elements 
in the other cluster (and similarly for the vertex $i$ in which $u$ assumes 
the maximum). 

\begin{center}
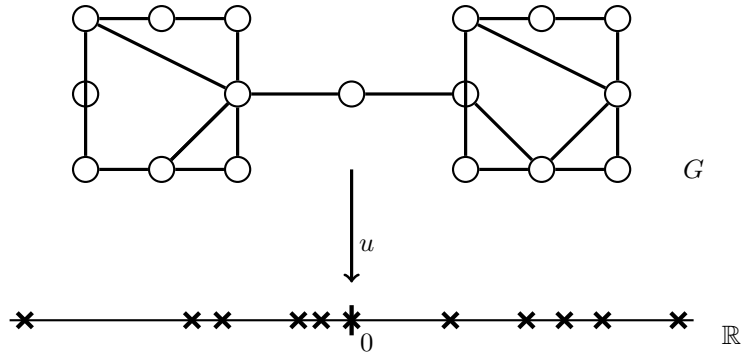
\begin{figure}[h!]
\begin{tikzpicture}[scale=1]
\begin{scope}[every node/.style={circle,thick,draw}]
    \node (A) at (0,0) {};
    \node (B) at (1,0) {};
    \node (D) at (1,2) {};
    \node (F) at (-1,2) {};
    \node (G) at (-1,0) {};
    \node (C) at (1,1) {};
    \node (E) at (0,2) {};
    \node (H) at (-1,1) {};

    \node (K) at (2.5,1) {};

    \node (A1) at (5,0) {};
    \node (B1) at (6,0) {};
    \node (D1) at (6,2) {};
    \node (F1) at (5-1,2) {};
    \node (G1) at (5-1,0) {};
    \node (C1) at (6,1) {};
    \node (E1) at (5,2) {};
    \node (H1) at (5-1,1) {};
\end{scope}

\begin{scope}[every edge/.style={draw=black,very thick}]
    \path [] (A) edge node {} (B);
    \path [] (B) edge node {} (C);
    \path [] (D) edge node {} (C);
    \path [] (D) edge node {} (E);
    \path [] (F) edge node {} (E);
    \path [] (F) edge node {} (G);
    \path [] (G) edge node {} (A);

    \path [] (C) edge node {} (K);   
    \path [] (H1) edge node {} (K);   
     \path [] (A) edge node {} (C);
    \path [] (F) edge node {} (C);

    \path [] (A1) edge node {} (B1);
    \path [] (B1) edge node {} (C1);
    \path [] (D1) edge node {} (C1);
    \path [] (D1) edge node {} (E1);
    \path [] (F1) edge node {} (E1);
    \path [] (F1) edge node {} (G1);
    \path [] (G1) edge node {} (A1);

     \path [] (A1) edge node {} (C1);
    \path [] (F1) edge node {} (C1);
     \path [] (A) edge node {} (C);
    \path [] (H1) edge node {} (A1);
\end{scope}
\draw [->, very thick] (2.5, 0) -- (2.5, -1.5);
\node at (2.7, -1) {$u$};
\node at (7.5, -2.2) {$\mathbb{R}$};
\node at (7, 0) {$G$};
\draw [thick] ( -2,-2) -- (7,-2);
\draw [ultra thick] (2.5, -1.8) -- (2.5, -2.2);
\node at (2.7, -2.3) {$0$};
\draw [ultra thick] (2.4, -1.9) -- (2.6, -2.1);
\draw [ultra thick] (2.6, -1.9) -- (2.4, -2.1);

\draw [ultra thick] (2, -1.9) -- (2.2, -2.1);
\draw [ultra thick] (2.2, -1.9) -- (2, -2.1);
\draw [ultra thick] (1.7, -1.9) -- (1.9, -2.1);
\draw [ultra thick] (1.9, -1.9) -- (1.7, -2.1);
\draw [ultra thick] (0.7, -1.9) -- (0.9, -2.1);
\draw [ultra thick] (0.9, -1.9) -- (0.7, -2.1);
\draw [ultra thick] (0.3, -1.9) -- (0.5, -2.1);
\draw [ultra thick] (0.5, -1.9) -- (0.3, -2.1);
\draw [ultra thick] (3.7, -1.9) -- (3.9, -2.1);
\draw [ultra thick] (3.9, -1.9) -- (3.7, -2.1);
\draw [ultra thick] (4.7, -1.9) -- (4.9, -2.1);
\draw [ultra thick] (4.9, -1.9) -- (4.7, -2.1);
\draw [ultra thick] (5.7, -1.9) -- (5.9, -2.1);
\draw [ultra thick] (5.9, -1.9) -- (5.7, -2.1);
\draw [ultra thick] (5.4, -1.9) -- (5.2, -2.1);
\draw [ultra thick] (5.2, -1.9) -- (5.4, -2.1);
\draw [ultra thick] (-1.7, -1.9) -- (-1.9, -2.1);
\draw [ultra thick] (-1.9, -1.9) -- (-1.7, -2.1);
\draw [ultra thick] (6.7, -1.9) -- (6.9, -2.1);
\draw [ultra thick] (6.9, -1.9) -- (6.7, -2.1);
\end{tikzpicture}
\caption{A graph $G$ and the embedding $u:V \rightarrow \mathbb{R}$ generated by the first eigenfunction of the Graph Laplacian $\mathcal{L}$.}
\end{figure}
\end{center}

\textbf{Example.}  Before stating the main result, we illustrate this notion by giving an example: we take all handwritten digits 0 and 1 from the MNIST data set.
Figure~\ref{fig:specembed} shows a spectral embedding into two dimensions: 
we have selected 8 specific points in the embedding and 
plot them in the corresponding
order below. As can be observed, both digits are highly clustered and there is 
very thin bottleneck (comprised of samples of `0'). 
The
samples with the smallest and largest $x-$coordinate in the embedding are both far away from the bottleneck.
These samples are the form of these digits that is least likely to be misclassified. 
We observe that if one writes a `0' in a very narrow way, there is a 
chance of it looking a lot like
a `1'. 
The left-most digit `0' with a little twirl on top is guaranteed to be a `0' 
because the twirl could not be explained by someone writing a `1'. 
Likewise, the farthest
digit `1', written at a 45 degree angle, is likely not a narrow `0'.

\begin{center}
\begin{figure}[h!]
\includegraphics[width=1.1\textwidth]{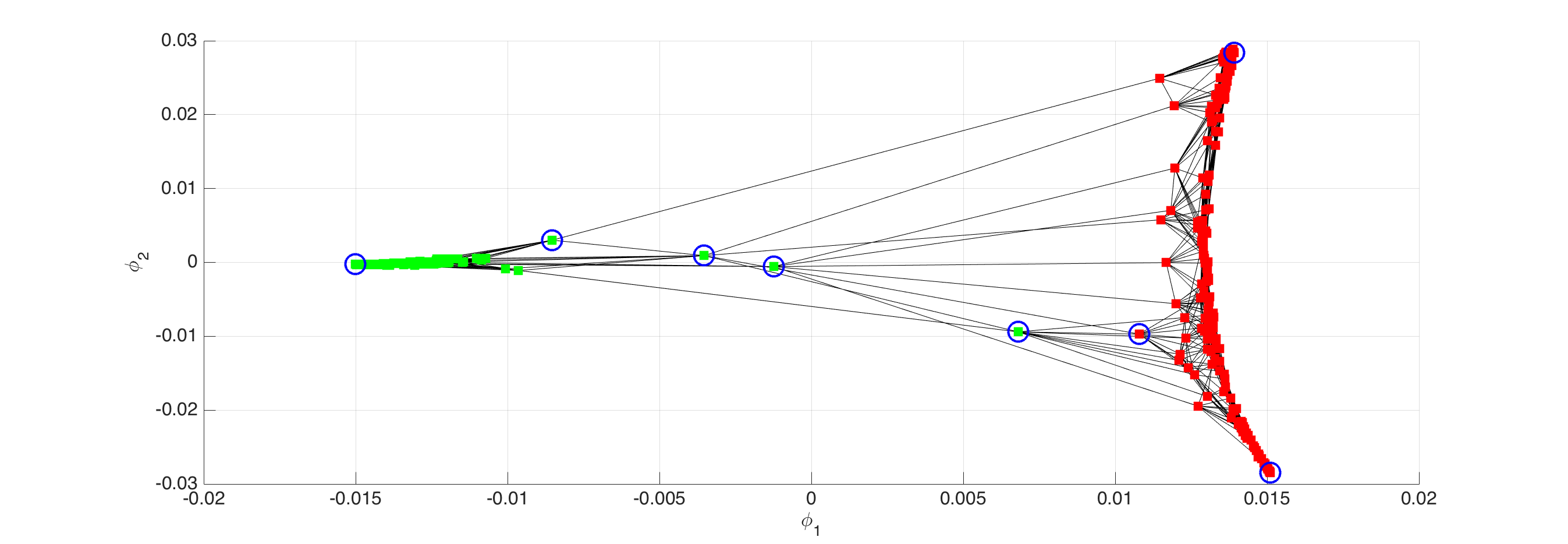}
\includegraphics[width=1.1\textwidth]{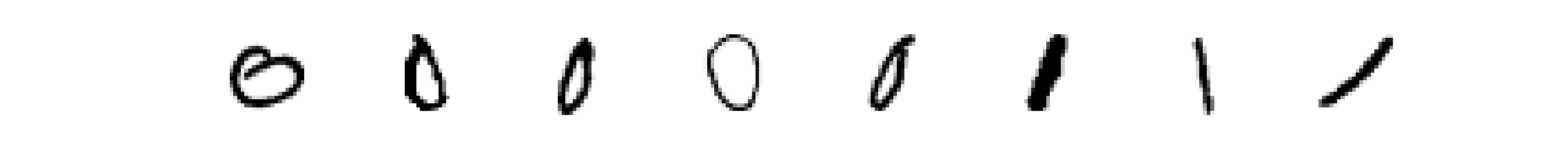}
\captionsetup{width = 0.95\textwidth}
\caption{A spectral embedding into two dimensions clearly reveals two clusters (above). 8 of these points are marked
(below) and we observe that both the extremal cases are far away from the bottleneck (and represent the most `typical' examples
if 0's and 1's).}
\label{fig:specembed}
\end{figure}
\end{center}

\subsection{Main result.} 
We now state the main result of the paper: vertices at which an eigenvector of the Laplacian assumes large values cannot be 
close to the set where the eigenfunction is small.

\begin{theorem} Suppose $\mathcal{L} u = \lambda u$ and that $\varepsilon > 0$ is so large that
$$ B = \left\{ i \in V:    - \varepsilon   \leq u(i) \leq \varepsilon   \right\} \neq \emptyset.$$
Then, for all $i \in V$,
$$ d_{B}(i) \log{\left(  \frac{1}{|1-\lambda|} \right) } 
\geq    \log{\left( \frac{ |u(i)| }{\|u\|_{L^{\infty}}}  \right)}  -  
\log{\left(\frac{1}{2} + \varepsilon\right)}.$$
\end{theorem}

Put differently, vertices $i \in V$ for which $|u(i)|$ is not too small are
far away from the bottleneck $B$ for which the spectral embedding vector is ambiguous. Some remarks are in order.\\

\textit{Remarks.}
\begin{enumerate}
\item The result is sharp for $\varepsilon = 0$ (see below for an example).
\item  There are obvious connections to the notion of eigenvector centrality which proposes that the importance of a point in a network can be 
measured by the $|u|$. 
Moreover, according to this heuristic, the point at
which the maximum is assumed is a good candidate for the most `central' 
point in the network. 
Our result implies that more `central' points with respect to this notion 
are located deep inside their respective cluster.
\item We observe that $d_B(i)$ seems to be a remarkably good approximation of $|u(i)|$: indeed,
we believe that understanding the precise relationship between the two objects could be of significant interest. This naturally relates to the
continuous case, where the mean first exit time gives rise to the Filoche-Mayboroda landscape function \cite{arnold, fil, stein3}. Our notion
of diffusion time may be understood as median first exit time. 
\item The constant $1/2$ is not special: one could generalize the
diffusion distance $d_B^{(p)}$ by requiring that $0 < p< 1$ of all Brownian motions have visited $B$; this gives rise to a different inequality with $1/2$
replaced by $1-p$.
\item The result is not restricted to the first eigenvector. Note that while the set $B$ 
depends on the eigenvalue,
its applicability is not restricted to graphs having exactly two clusters.
\end{enumerate}

\subsection{Absorbing states.} The purpose of this section is to give a related result in the special case of the random walk having absorbing states: assume
the edge weights of $G=(V,E)$ are such that there is an absorbing set of vertices $\partial V \subset V$ with the property that every random
walk gets eventually absorbed with likelihood 1. Then the first eigenfunction of $\mathcal{L}$ is intimately linked to absorbtion time (and, more
generally, the diffusion distance $d_{\partial V}$ and the first eigenfunction seem to have very strong correlation, see \S 5).
\begin{theorem} Suppose $\mathcal{L} u = \lambda_1 u$ 
and $u\big|_{\partial V} = 0$,
then
$$ d_{\partial V}(i)  \log{\left(  \frac{1}{|1-\lambda_1|} \right) }  
\geq \log{\left( \frac{2|u(i)|}{\|u\|_{L^{\infty}}}  \right)}.$$
\end{theorem}

We observe that $|u(i)| > \|u\|_{L^{\infty}}/2$ is required for the result to be nontrivial. The inequality is asymptotically sharp: 
let us consider a complete graph $K_{n}$ with all weights being identically $p_{ij} = 1/n$ and one vertex chosen to be the absorbing state.

\begin{center}
\begin{figure}[h!]
\begin{tikzpicture}[scale=1]
\begin{scope}[every node/.style={circle,thick,draw}]
    \node (A) at (0,0) {};
    \node (B) at (1,0) {};
    \node (C) at (0,1) {};
    \node (D) at (1,1) {};
    \node (E) at (2.5,0.5) {};

\end{scope}

\begin{scope}[every edge/.style={draw=black,very thick}]
    \path [] (A) edge node {} (B);
    \path [] (A) edge node {} (C);
    \path [] (A) edge node {} (D);
    \path [] (B) edge node {} (C);
    \path [] (B) edge node {} (D);
    \path [] (C) edge node {} (D);
    \path [] (A) edge node {} (E);
    \path [] (B) edge node {} (E);
    \path [] (C) edge node {} (E);
    \path [] (D) edge node {} (E);

\end{scope}

\draw [ultra thick] (2.2, 0.2) -- (2.8, 0.2);
\draw [ultra thick] (2.8, 0.8) -- (2.8, 0.2);
\draw [ultra thick] (2.2, 0.2) -- (2.2, 0.8);
\draw [ultra thick] (2.2, 0.8) -- (2.8, 0.8);

\node at (2.8, -0.1) {$\partial V$};

\end{tikzpicture}
\caption{Complete graphs show the result to be tight.}
\end{figure}
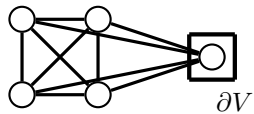
\end{center}

An easy computation shows that the constant vector satisfies
$\mathcal{L}u = \frac{1}{n}u$ and $d_{\partial V}(i)$ 
is the smallest integer with 
$$\left(1 - \frac{1}{n} \right)^{d_{\partial V}(i)} \leq \frac{1}{2}.$$ 
Since
$$ \lim_{n \rightarrow \infty}{ \left(1 - \frac{1}{n} \right)^{n}} = \frac{1}{e}, \quad \mbox{we get that} \quad \lim_{n \rightarrow \infty}{\frac{d_{\partial V}(i)}{n}} = \log{2} \, .$$
Thus, for $n$ large,
$$ d_{\partial V}(i) \log{\left(  \frac{1}{1-\lambda} \right) } \sim  \log{\left(1 + \frac{1}{n-1} \right)}n\log{2} \sim \log{2}.$$
The very same example can be used to show that the main result is sharp in the case, where the eigenfunction actually
has a root $u(i) = 0$. This can be achieved by taking two separate copies of complete graphs $K_n$, adding a free vertex and connecting
the free vertex to all other vertices; a simple computation shows that this reduces the main result to Theorem 2, which
is sharp.

\subsection{Related results.} 
This paper is inspired by similar results in the continuous analogue due 
to the second and third author \cite{manas} 
(earlier results in a similar spirit were given by Georgiev \& 
Mukherjee \cite{georg} and the third author \cite{stein1}).  
Bovier, Eckhoff, Gayrard \& Klein \cite{bov} 
study closely related questions regarding metastable states in 
Markov chains with bottlenecks and 
their relation to the spectrum and distribution of exit times
(see also \cite{bouv1, bouv2}).
For other approaches towards understanding
the success of spectral clustering we refer to
Meila \& Shi \cite{meila} and Ng, Jordan \& Weiss 
\cite{jordan}. 
In this context, we especially emphasize the results of Gavish \& Nadler \cite{gavish}, 
who study the relation between the exit times of diffusion and the normalized cut.

\section{Proofs}

\subsection{$\mathcal{L}$ and random walks.}
First, we note that the spectrum of $\mathcal{L}$ satisfies:
$$ \sigma(\mathcal{L}) \subset \left\{z \in \mathbb{C}: |1-z| \leq 1 \right\}.$$
This follows trivially from the Gershgorin circle theorem since, 
$$ \mathcal{L}_{ii} = 
1 - p_{ii} \, , \, \text{and } 
\sum_{\substack{j=1\\j \neq i}}^{n} \mathcal{L}_{ij} = 1 - p_{ii} \, .$$
We quickly describe the underlying connection between $\mathcal{L}$ and 
random walks on $G$,
this connection will be a crucial tool for all subsequent proofs. 
Fix a vertex $i \in V$, let $x_0 = i$ and as before let $x_{n}$ for
$n \in \mathbb{N}$, 
denote the random walk associated with a Markov chain on the graph.
By definition of the Graph Laplacian
\begin{align*}
 \mathbb{E}(u(x_{n+1})\big| x_n ) &= 
 \sum_{j \in V}{ \mathbb{P}\left(x_{n+1} = j\right) u(j)} = 
 \sum_{j \in V}{ p_{ij} u(j)} \\
&=u(x_n) +  \sum_{j \in V}{ p_{ij} (u(j) - u(x_n))} \\
&= u(x_n) - \left( \mathcal{L}u \right)(x_n)
\end{align*}
If $\mathcal{L}u = \lambda u$, then we get
$$ \mathbb{E}(u(x_{n+1}) \big| x_n)  = (1-\lambda) u(x_n)$$
and, by induction,
$$ \mathbb{E}(u(x_{n})) = (1- \lambda)^n u(x_0).$$

\subsection{Proof of Theorem 1}
\begin{proof} Let us assume that $u$ solves $\mathcal{L}u = \lambda u$, normalized to $\|u\|_{L^{\infty}} = 1$, let $i \in V$ be arbitrary and assume w.l.o.g. (after possibly replacing $u$ by $-u$)  that $0 < u(i)  < 1$.
As before, we start random walks in $x_0 = i$ whose transition
probabilities  are given by $p_{ij}$. 
We see that, 
$$ \mathbb{E}(u(x_{n})) = (1- \lambda)^n u(x_0).$$
We fix the value $n = d_{B}(i)$ and make a distinction between those random walks that never enter the set $B$ up to time $n$ (which we call event $A$)
and those random walks that are contained in $B$ at some point (event $\neg A$). 
We have
$$ |\mathbb{E}(u(x_n))| =  |\mathbb{E}(u(x_n) \big| A)|\cdot \mathbb{P}(A) +  |\mathbb{E}(u(x_n)\big| \neg A)| \cdot \mathbb{P}(\neg A).$$  
Trivially,
$$  |\mathbb{E}(u(x_n) \big| A)| \mathbb{P}(A) \leq  \mathbb{P}(A).$$
In the second case, the random walk entering the set $B$ at some point $x_k$, we can employ the Markovian property and conclude that
$$ |\mathbb{E}(u(x_n)\big| \neg A)| \mathbb{P}(\neg A) \leq \varepsilon   |1- \lambda|^{n-k} \mathbb{P}(\neg A) \leq \varepsilon \mathbb{P}(\neg A).$$
Altogether, this implies
\begin{align*}
u(i) |1- \lambda|^n &=  u(i) |1- \lambda|^n u(x_0) = \mathbb{E}(u(x_{n})) \\
&\leq   \mathbb{P}(A)  + \varepsilon   \mathbb{P}(\neg A).
\end{align*}
It follows from the definition of diffusion time that
$$ \mathbb{P}(\neg A) \geq \frac{1}{2} \, .$$
Thus
$$ u(i) |1- \lambda|^{d_B(i)} \leq \frac{1}{2} + \varepsilon$$
from which the result follows.
\end{proof}

\subsection{Proof of Theorem 2.}
\begin{proof}
 Assume that the eigenfunction is normalized as $\|u\|_{L^{\infty}} = 1$ and let $i \in V$ such that $1/2 < u(i) < 1$
and let $n = d_{\partial V}(i)$. Then, we can analyze the expectation in
$$ (1-\lambda)^n u(i) = \mathbb{E}(u(x_{n}))$$
by concluding that at least in half of all cases we get 0 (this follows from the definition of the diffusion distance) -- in the other cases, we do not know what to expect but the contribution
can certainly not be larger than $u(i)$ since this is the maximal value; therefore
$$  |1-\lambda|^n u(i) \leq   |(1-\lambda)^n u(i)| = |\mathbb{E}(u(x_{n}))| \leq \frac{1}{2}$$
and this implies the statement. 
\end{proof}

\section{The continuous case: hot spots}
These results have a continuous equivalent that may be of independent interest and has some applications to the hot spots problem. Let $\Omega \subset \mathbb{R}^2$ be a convex set with smooth boundary and assume
\begin{align*}
-\Delta u &= \lambda u \qquad\mbox{in}~ \Omega \\
\frac{\partial u}{\partial n} &= 0  \qquad \mbox{on}~\partial \Omega,
\end{align*}
where $\lambda > 0$ is assumed to be a nontrivial eigenvalue. Note that
$$\int_{\Omega}{u(x)~dx} = -\frac{1}{\lambda}\int_{\Omega}{\Delta u(x)~dx}  = - \frac{1}{\lambda} \int_{\partial \Omega}{\frac{\partial u}{\partial \nu}~d\sigma} =  0,$$ the nodal set is not the empty set $$\mathcal{N} = \left\{x \in \Omega: u(x) = 0\right\} \neq \emptyset.$$
We will now show that both maximum and minimum are both a nontrivial distance away from the nodal set. The technique is given
by a variant of the argument used in \cite{manas}. 

\begin{theorem} 
Let $\Omega \subset \mathbb{R}^n$ be bounded with smooth boundary. 
Suppose $u$ assumes its global maximum or minimum at $x \in \Omega$. 
Let $d_{\mathcal{N}}(x)$ denotes the expected time for half of all 
Brownian motions started in $x$ and
reflected off the boundary $\partial \Omega$ to hit $\mathcal{N}$. 
Then $$ d_{\mathcal{N}}(x) \lambda \geq \log{2}.$$
\end{theorem}

It is easy to see that this is the correct scaling: consider the eigenfunction $u(x) = \cos{(n x)}$ on $\Omega = [0, \pi]^2$. The eigenvalue is $\lambda = n^2$, the
extrema are a Euclidean distance $\sim n^{-1}$ away from the set $\mathcal{N}$. The appropriate scaling for a significant fraction of Brownian motion started
in the extrema to hit $\mathcal{N}$ is $\sim n^{-2}$. Under additional assumptions on domain and eigenfunction, it is possible to recover information about the
Euclidean distance; more precisely, for convex $\Omega$ and $\lambda$ being the first eigenvalue, it is possible to obtain a result along the lines of
$$ d_{\mathcal{N}}(x) \sim \mbox{dist}(x, \partial \Omega)^2$$
with implicit constants depending only on the dimension (see the proofs for details).
Melas \cite{melas},
proving a conjecture of Payne \cite{payne}, has shown that the first nontrivial Laplacian eigenfunction with Neumann boundary conditions on a convex domain
with $C^{\infty}-$boundary has a nodal line that intersects the boundary in exactly two points (splitting the domain). 
We cannot specify in advance where maximum or minimum is going to occur, the inequality allows us to specify a subregion in which they must lie (see Fig. \ref{bone-dis}).
\begin{center}
\begin{figure}[h!]
\includegraphics[width=0.85\textwidth]{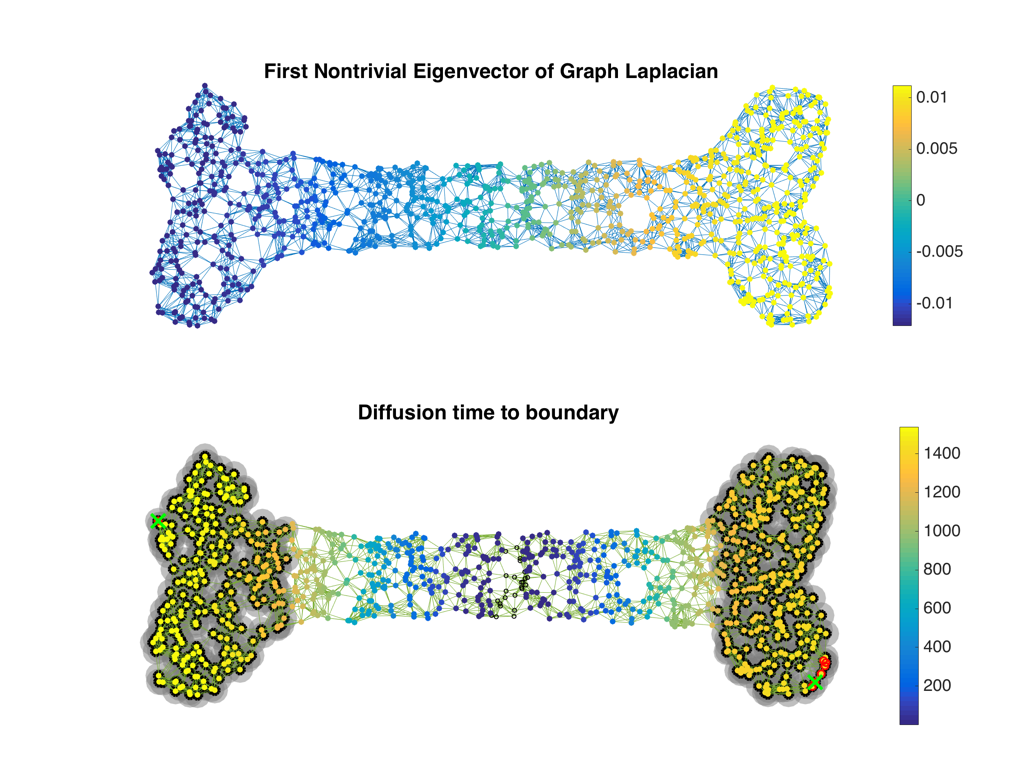}
\captionsetup{width = 0.95\textwidth}
\caption{
\label{bone-dis}
The maximum of the first eigenfunction $v_1$ is indicated by the green cross, and those of the diffusion time by red circles, in both left and right sides. On both sides, the two maxima coincide. 
The set $\left\{ i \in V: -d_{B}(i) \log{\left( |1-\lambda_1| \right) } \geq    \log{\left( 2\right)} \right\}$ is drawn in yellow.
}
\end{figure}
\end{center}

\textbf{Example.} 
Consider the discrete pseudo-planar example shown in Figure~\ref{bone-dis}. 
Here, $n=1000$ nodes are randomly sampled within a domain and a 
graph is built by connecting 10 nearest neighbors of each node and 
symmetrizing the adjacency matrix.
The first (nontrivial) eigenvector of the graph Laplacian achieves maximum and minimum on the two ends of the domain. 
On the right side, we obtain $\max d_{B} = 1428$ while our 
lower bound predicts that the diffusion time in the minimum is 
at least $1201$. 
On the left-hand side, we obtain
$\max d_{B} = 1540$ and are guaranteed at least  $1246$ from our 
lower bound. 
It is remarkable that the correlation between the absolute value of 
the first eigenfunction and $d_B$ is 
as high as 0.9866: the diffusion 
time to the boundary (here: bottleneck) is a very good approximation 
of the first eigenfunction.

\subsection{Proof of Theorem 3}
\begin{proof} We will essentially repeat the argument from the proof of Theorem 2 in spirit. However, while it is certainly possible to approximate the continuous
object by a graph and conclude Theorem 3 directly from Theorem 2, we wish to avoid certain technicalities involved with that and will give a fully continuous
proof. Let $\Omega \subset \mathbb{R}^n$ with smooth boundary and assume $u \neq 0$ satisfies
\begin{align*}
-\Delta u &= \lambda u \qquad\mbox{in}~ \Omega \\
\frac{\partial u}{\partial n} &= 0  \qquad \mbox{on}~\partial \Omega,
\end{align*}
where $\lambda > 0$ is assumed to be a nontrivial eigenvalue. The mean value of the function is 0, which implies that the nodal set
$$ \mathcal{N} = \left\{x \in \Omega: u(x) = 0\right\}$$
is not empty.  Since $\partial \Omega$ is assumed to be smooth, we can use the Feynman-Kac formula
for the Neumann problem \cite{hsu, stein1} and write, for every $t \geq 0$,
$$ u(x) = e^{\lambda t} ~\mathbb{E}_{x} \left( u(\omega(t))   \right),$$
where $\mathbb{E}_x$ is taken with respect to all Brownian motion started in $x$, running until time $t$ and reflected
off the boundary $\partial \Omega$. Let us now assume that $x\in \Omega$ is the 
location of the maximum (a similar argument holds for the minimum). 
It is not difficult to see that the solution of the equation
restricted to the connected component of $\left\{x \in \Omega: u(x) > 0\right\}$ (which we denote by $\Omega_1$) 
also satisfies the equation
\begin{align*}
-\Delta u &= \lambda u \qquad\mbox{in}~ \Omega_1 \\
u &= 0 \qquad\mbox{in}~ \mathcal{N} \cap \partial \Omega_1 \\
\frac{\partial u}{\partial n} &= 0  \qquad \mbox{on}~\partial \Omega_1\setminus \mathcal{N}.
\end{align*}
We can now use the probabilistic interpretation with respect to this 
new problem: $\omega(t)$ will denote a Brownian
motion started in $x$ and running for time $t$ that is absorded on $\mathcal{N} \cap \partial \Omega_1$ and reflected
on $\partial \Omega_1\setminus \mathcal{N}$. We denote the probability of such a Brownian motion impacting on
 $\mathcal{N} \cap \partial \Omega_1$ within time $t$ as $p_t(x)$.
It is now easy to see that
$$
 u(x) = e^{\lambda t} ~\mathbb{E}_{x} \left( u(\omega(t))   \right)  \leq  e^{\lambda t} u(x) p_t(x).
$$
If we set $t = d_{\mathcal{N}}(x)$, then -- by definition -- $p_t(x) = 1/2$ and thus
$$2 \leq e^{\lambda \cdot d_{\mathcal{N}}(x)},$$
which is the desired statement.
\end{proof}

The reason why this argument does not immediately translate into a statement for the Euclidean distance $\mbox{dist}(x, \partial \Omega)$
is that Brownian motion is reflected 
on $\partial \Omega_1 \setminus \mathcal{N}$. For complicated 
labyrinth-type domains, it is 
possible for the maximum to be assumed in a point that is very close to $\mathcal{N}$ with respect to Euclidean distance but far away
in terms of diffusion distance. 
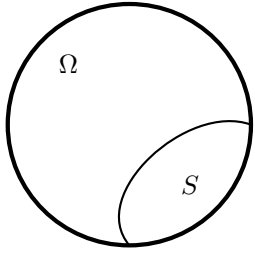
\begin{figure}[h!]
\begin{tikzpicture}[scale=0.8]
\draw [ultra thick] (2,2) circle (2cm);
    \draw [thick] (2,0) to[out=130,in=160] (4,2);
\node at (3,1) {$S$};
\node at (1,3) {$\Omega$};
\end{tikzpicture}
\captionsetup{width = 0.95\textwidth}
\caption{For convex $\Omega \subset \mathbb{R}^n$, every subset $S \subset \Omega$ has to have a large boundary $\partial S \cap \Omega$ with large $(n-1)-$dimensional volume unless $S$ or $\Omega \setminus S$ is small.}
\end{figure}
If $\Omega$ is convex and we consider, say, the first (nontrivial) eigenfunction $u$, then this scenario can be excluced: by considering $S = \left\{ x \in \Omega: u(x) > 0\right\}$,
it becomes clear that we need to ensure that $\partial S \cap \Omega$ cannot be too small. 
A result of the third author \cite{stein2} 
(refining a result of Dyer \& Frieze \cite{dyer}) states
that for open and convex $\Omega \subset \mathbb{R}^n$ and all open 
subsets $A \subset \Omega$
$$\mathcal{H}^{n-1}\left(\partial A \cap \Omega\right) \geq \frac{4}{\diam(\Omega)}\frac{|A||\Omega \setminus A|}{|\Omega|},$$
where $\mathcal{H}^{n-1}$ denotes the $(n-1)-$dimensional Hausdorff measure. An easy compactness argument shows that for the first 
eigenfunction $|S| \gtrsim_n |\Omega|$ for some implicit 
constant depending only on the dimension from
which the equivalence of $d_{\mathcal{N}}$ and $\mbox{dist}(\cdot, \mathcal{N})^2$
(up to constants) follows.

\section{Variations, remarks and comments}
Since the argument itself is rather elementary, it is not surprising that there should be a series of natural variations.
The purpose of this section is to outline some of them and remark on some additional interesting features.

\subsection{Sharpness of the inequality.} We observe that the result is close to sharp in a variety of different settings (this includes
both relatively sparse and relatively dense graphs). \\

\textit{Small world graphs.} 
Consider a small world graph with 128 nodes on a ring. 
The absorbing boundary, $\partial V$ is 8 randomly selected vertices.
Furthermore, additional edges are 
generated between any two vertices in an i.i.d. way, such that the 
expected number of additional edges is 64. 
A typical realization is shown in Figure \ref{fig:sw-sparse}.
We have $\max d_{\partial V} = 37$ 
while the inequality predicts at least $24$.  
The first eigenfunction is a very good
approximation of $d_{\partial V}$; their correlation is 0.9862.
\begin{center}
\begin{figure}[h!]
\includegraphics[width=\textwidth]{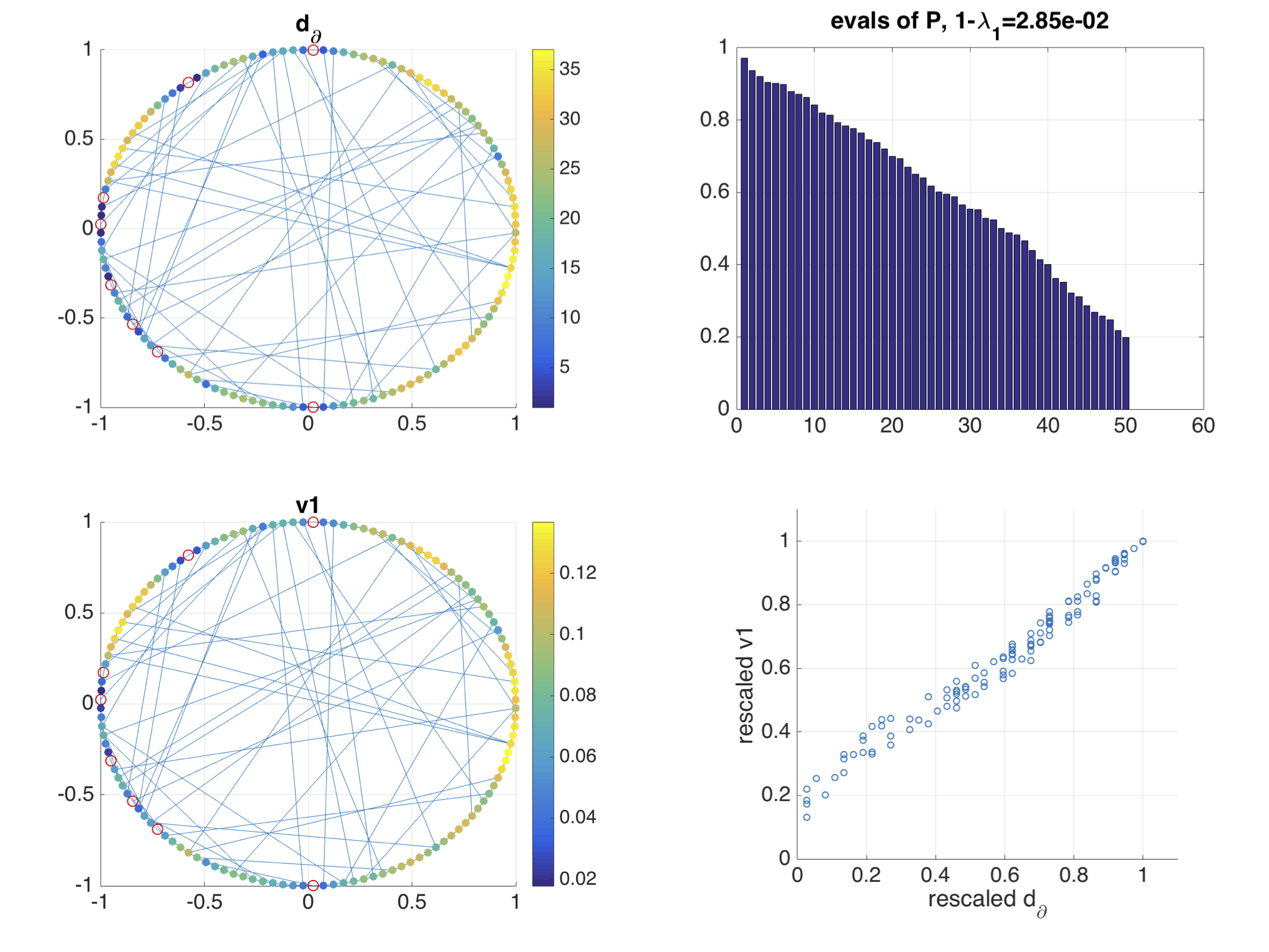}
\captionsetup{width=0.95 \textwidth}
\caption{\label{fig:sw-sparse}
The diffusion time to boundary $d_{\partial V}$ (top) and the first eigenvector $v_1$ (bottom) on a small world graph
The scatter plot (bottom right) shows $v_1$ and $d_{\partial V}$ rescaled to have maximum 1, respectively.
}
\end{figure}
\end{center}

\textit{Dense graphs.} The second example, shown in Fig. \ref{fig:sw-dense}, has the same setup but a larger number of connections.
We find that $\max d_{\partial V} = 15$ while the inequality predicts at least $13$. The correlation between the first eigenfunction and $d_{\partial V}$
is 0.9897.
\begin{center}
\begin{figure}[h!]
\includegraphics[width=\textwidth]{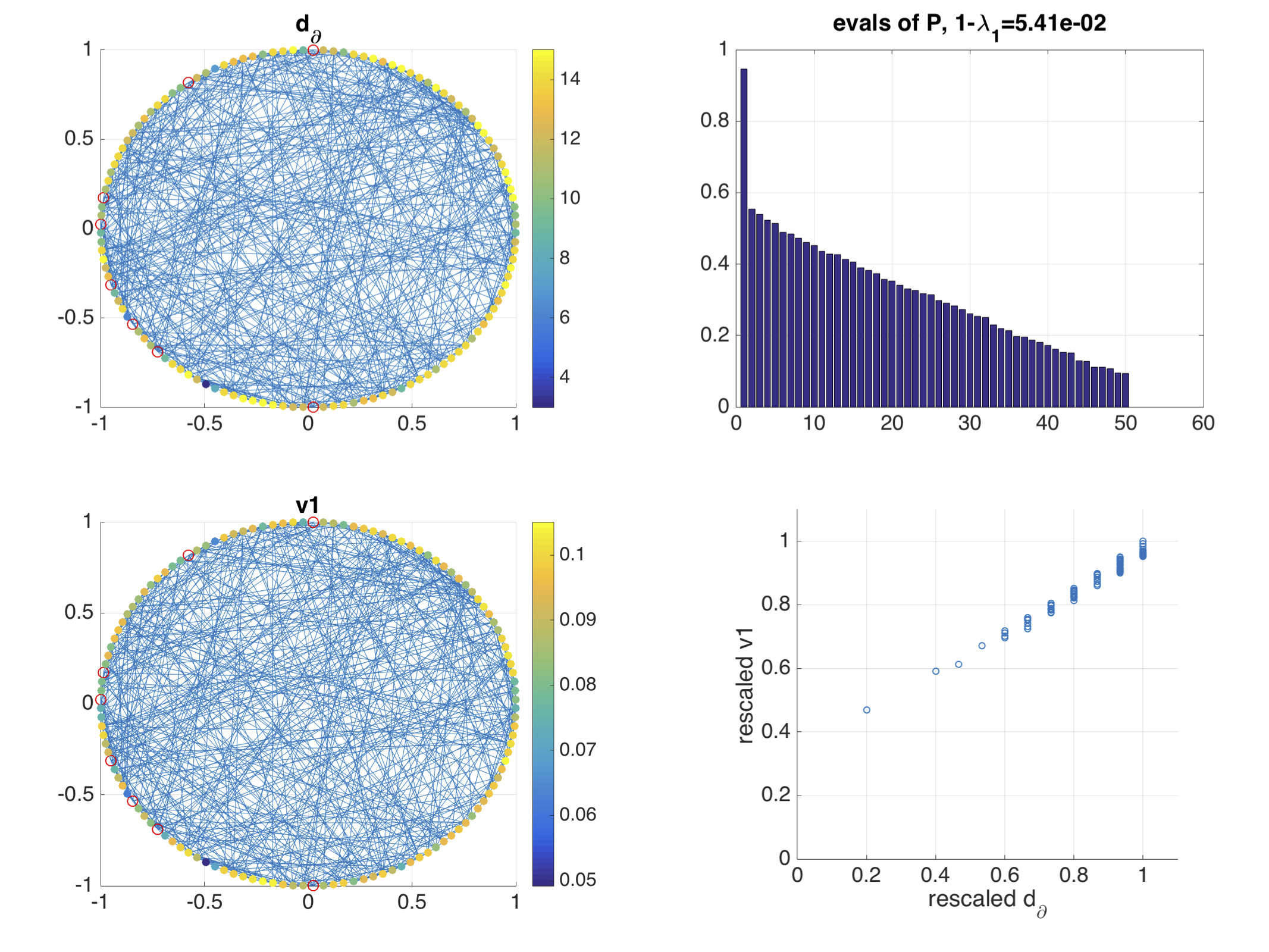}
\caption{\label{fig:sw-dense}
The first eigenvector and diffusion time to boundary $d_{\partial V}$ on another graph.
}
\end{figure}
\end{center}
\vspace{-10pt}
These examples motivate the following question.
\begin{quote} \textbf{Problem.} 
To what extent does $d_{\partial V}$ approximate the first eigenfunction?
\end{quote}
The examples considered above show a very large correlation. It is not difficult to see that correlation by itself is not the right notion to use since
one can construct examples in which the first eigenfunction localizes to a much stronger extent than $d_{\partial V}$, which is large in a variety
of different places -- however, in that case we would still expect $d_{\partial V}$ to approximate the first eigenfunction very well on the domain
where it is large (and, in particular, a large correlation provided we only compute the correlation on that domain).
We observe that similar connections exists in the continuous setting (see \cite{fil, manas, stein3}) 
but are not fully understood there either.

\subsection{Families of extremizers.} 
As was discussed after stating Theorem 2, the complete graph $K_n$ with one vertex selected as boundary
point shows that the constant $\log{2}$ cannot be improved (by letting $n \rightarrow \infty$). The purpose of
this section is to show that there is a much larger family of graphs for which this is the case. 

\begin{proposition} 
Let $G_n = (V_n, E_n)_{n=1}^{\infty}$ be a sequence of graphs and 
let $\partial V_{n}$ be a single vertex which is an absorbing state. 
Furthermore, assume that
$$ \forall~ i \in V_n\setminus \partial V_n: ~\sum_{i \sim j \atop j \in \partial V_n}{p_{ij}} \quad \mbox{is independent of}~i~\mbox{and tends to 0 as}~n \rightarrow \infty.$$
Then there exists an eigenfunction with eigenvalue $\lambda_n$ 
with maximum at $i_n \in V_n$ such that
$$ d_{\partial V}(i_n) \log{\left( \frac{1}{1 - \lambda_n}\right)}  \rightarrow \log{2}.$$
\end{proposition}
\begin{proof} It is easy to see that under these assumptions the constant vector
is an eigenvector. The eigenvalue is given as
$$ \lambda =  \sum_{i \sim j \atop j \in \partial V_n}{p_{ij}} -1,$$
where the value of $i$ is not important as the sum does, by assumption, not depend on that parameter. The diffusion time to the boundary is given
as
$$ \min \left\{k \in \mathbb{N}: \left(1- \sum_{i \sim j \atop j \in \partial V_n}{p_{ij}}  \right)^k \leq \frac{1}{2} \right\} = \left\lceil \frac{-\log{2}}{\log{\left(1- \sum_{i \sim j \atop j \in \partial V_n}{p_{ij}} \right)}} \right\rceil$$
from which the statement follows as long as the quantity tends to infinity, which requires
$$ \sum_{i \sim j \atop j \in \partial V_n}{p_{ij}} \rightarrow 0.$$
\end{proof}
This makes it quite easy to construct infinite families of graphs for which equality is asymptotically attained. One possible example
is given by the circle graph $C_n$ if we add an additional vertex, which we define to be the boundary, and impose transition probabilities
to be
$$ \mathbb{P}\left(  \mbox{interior point} \rightarrow \mbox{boundary point}   \right) = \varepsilon$$
for some $\varepsilon$ small; the transition probabilities within $C_n$ can be chosen in an arbitrary fashion (as long as they add up to $1-\varepsilon$).

\subsection{Schr\"odinger-type equations.} 
Akin to arguments in \cite{manas}, 
we can extend our main result to more general equations
of the type
$$ \mathcal{L}u = Wu,$$
where $W=\mbox{diag}(w_1, \dots, w_n)$ is assumed to be a diagonal matrix; 
the role of the `eigenvalue' is then played
by the supremum-norm of the diagonal 
$$ \|W\|_{L^{\infty}} = \max_{}{|w_i|} \, .$$
The setup coincides with the eigenvalue case, if $W$ is a multiple of the identity.
We observe that under the
assumption that the graph is connected and the absorbing set is nonempty $\partial V \neq \emptyset$, then the diffusion
process induced by $\mathcal{L}$ ultimately transports mass to the boundary and therefore, for every vector $v$,
$$ \lim_{k \rightarrow \infty}{ \mathcal{L}^k v}  = 0.$$
Thus, the equation $\mathcal{L}u = Wu$ implies that $\|W\|_{L^{\infty}} < 1$.

\begin{corollary} 
Suppose $\mathcal{L} u = W u$ with $u\geq 0$ and $u|_{\partial V}= 0$. Then
$$ d_{\partial V}(i)  \log{\left(  \frac{1}{1- \|W\|_{L^{\infty}}} \right) }  \geq \log{\left( \frac{2|u(i)|}{\|u\|_{L^{\infty}}}  \right)}.$$
\end{corollary}
\begin{proof} The proof is essentially identical to the Proof of Theorem 2 once we use
$$ \mathbb{E}(u(x_{n+1})\big| x_n ) = u(x_n) - \mathcal{L} u(x_n) \leq u(x_n) - \|W\|_{L^{\infty}}u(x_n).$$
\end{proof}

There is also a natural variant of Theorem 1 that can be obtained via the same argument.\\

\textbf{Acknowledgement.} The authors are grateful to Raphy Coifman, Jianfeng Lu and Boaz Nadler for valuable discussions.


\vspace{-10pt}
\begin{thebibliography}{4}

\bibitem{arnold} D. Arnold, G. David, D. Jerison, S. Mayboroda, and M. Filoche. Effective confining potential of quantum states in disordered media. Physical Review Letters, 116 (5), 2016.

\bibitem{burd2} R. Banuelos and K. Burdzy, On the `hot spots' conjecture of J. Rauch. J. Funct. Anal. 164 (1999), 1--33. 
\bibitem{ban2} R. Banuelos and T. Carroll, Brownian motion and the fundamental frequency of a drum. Duke Math. J. 75 (1994), no. 3, 575--602. 
\bibitem{car2} R. Banuelos and T. Carroll, 
Addendum to: Brownian motion and the fundamental frequency of a drum,  Duke Math. J. 82 (1996), 227. 

\bibitem{bov}
A. Bovier, M. Eckhoff, V. Gayrard and M. Klein, 
Metastability and low lying spectra in reversible Markov chains. 
Comm. Math. Phys. 228 (2002), no. 2, 219--255. 

\bibitem{bouv1} A. Bovier, M. Eckhoff, V. Gayrard, and M. Klein,
Metastability in reversible diffusion processes. I. Sharp asymptotics for capacities and exit times. 
J. Eur. Math. Soc. (JEMS) 6 (2004), no. 4, 399--424. 

\bibitem{bouv2} A. Bovier, V. Gayrard and M. Klein,
Metastability in reversible diffusion processes. II. Precise asymptotics for small eigenvalues.
J. Eur. Math. Soc. (JEMS) 7 (2005), no. 1, 69--99.

\bibitem{burd} K. Burdzy and W. Werner, 
A counterexample to the ``hot spots'' conjecture. 
Ann. of Math. (2) 149 (1999), no. 1, 309--317. 

\bibitem{dyer} M. Dyer, A. Frieze, Computing the volume of convex bodies: a case where randomness provably helps, Proc. Sympos. Appl. Math., 44, Amer. Math. Soc., Providence, RI, 1991. 

\bibitem{fil} M. Filoche and S. Mayboroda, 
Universal mechanism for Anderson and weak localization. 
Proc. Natl. Acad. Sci. USA 109 (2012), no. 37, 14761-14766. 

\bibitem{gavish} M. Gavish and B. Nadler,
Normalized cuts are approximately inverse exit times. 
SIAM J. Matrix Anal. Appl. 34 (2013), no. 2, 757-772.

\bibitem{georg} B. Georgiev and M. Mukherjee, Nodal Geometry, Heat Diffusion and Brownian Motion, arXiv:1602.07110

\bibitem{hsu} P. Hsu, 
Probabilistic approach to the Neumann problem. 
Comm. Pure Appl. Math. 38 (1985), no. 4, 445--472. 

\bibitem{meila} M. Meila, J. Shi, A random walks view of spectral segmentation, AISTATS 2001 

\bibitem{melas} A. Melas,
On the nodal line of the second eigenfunction of the Laplacian in $\mathbb{R}^2$. 
J. Differential Geom. 35 (1992), no. 1, 255--263. 

\bibitem{jordan} A. Ng, M. I. Jordan and Y. Weiss, On spectral clustering: Analysis and an algorithm, Advances in Neural Information Processing Systems (NIPS) 14, 2002.

\bibitem{payne} L. Payne,  Isoperimetric inequalities and their applications.  SIAM Rev. 9 1967 453--488. 

\bibitem{manas} M. Rachh and S. Steinerberger, On the location of Maxima of Solutions of Schr\"odinger's equation, arXiv:1608.06604


\bibitem{stein1} S. Steinerberger, Lower bounds on nodal sets of eigenfunctions via the heat flow. Comm. Partial Differential Equations 39 (2014),  2240--2261. 

\bibitem{stein2} S. Steinerberger,
Sharp $L^1$-Poincar\'{e} inequalities correspond to optimal hypersurface cuts. Arch. Math. 105 (2015), no. 2, 179--188.

\bibitem{stein3} S. Steinerberger, Localization of Quantum States and Landscape Functions, Proc. Amer. Math. Soc., to appear

\end{thebibliography}
\end{document}